\documentclass[10pt,reqno]{amsart}
\usepackage{amssymb}
\usepackage{amsthm}

\usepackage{amssymb,amsmath,amsthm,amsxtra,amsfonts}
\usepackage{mathtools}

\newcommand{\bbC}{{\mathbb{C}}}

\newcommand{\bbR}{{\mathbb{R}}}

\newcommand{\calJ}{{\mathcal J}}

\newcommand{\bdnot}{{\boldsymbol{0}}}

\newcommand{\Arg}{\text{\rm{Arg}}}

\newcommand{\spann}{\text{\rm{span}}}
\newcommand{\rank}{\text{\rm{rank}}\,}

\newcommand{\beq}{\begin{equation}}
\newcommand{\eeq}{\end{equation}}
\newcommand{\ba}{\begin{align*}}
\newcommand{\ea}{\end{align*}}

\DeclareMathOperator{\re}{Re}
\DeclareMathOperator{\im}{Im}

\DeclareMathOperator*{\res}{Res}

\allowdisplaybreaks
\numberwithin{equation}{section}

\newtheorem{theorem}{Theorem}
\newtheorem{lemma}{Lemma}
\newtheorem{proposition}{Proposition}

\newtheorem{corollary}{Corollary}
\newtheorem{definition}{Definition}
\theoremstyle{remark}
\newtheorem*{remarks}{Remarks}
\newtheorem*{remark}{Remark}

\begin{document}

\title[Rank one non-Hermitian perturbations of Hermitian $\beta$-ensembles]{Rank one non-Hermitian perturbations of Hermitian $\beta$-ensembles  of random matrices}
\author{Rostyslav Kozhan}
\address{Rostyslav Kozhan\\
         Uppsala University\\
         Uppsala, Sweden}
\email{kozhan@math.uu.se}
\thanks{The research was partly funded by the grant KAW 2010.0063 from the Knut and Alice Wallenberg Foundation}

\begin{abstract}
For any $\beta>0$, we provide a tridiagonal matrix model and compute the joint eigenvalue density of a random rank one non-Hermitian perturbation of Gaussian and Laguerre $\beta$-ensembles of random matrices.
\end{abstract}

\maketitle

\section{Introduction}

The energy Hamiltonian of a closed quantum system is usually modelled by a Hermitian random matrix $H$. 
The Hamiltonian of this system after coupling it to the outer world via $s$ open channels is modelled in the physics literature by the so-called effective Hamiltonian\footnote{In the physics literature it is more common to take $H - i \Gamma$, which can be reduced to our case by a simple symmetry.}
\begin{equation}\label{eff}
H_{eff}=H + i \Gamma,
\end{equation}
where $\Gamma\ge \bdnot$ is a rank $s$ positive semi-definite Hermitian matrix that is independent of $H$. In this paper we are concerned with the exact joint eigenvalue distribution of~\eqref{eff} when there is one open channel ($\rank \Gamma= s=1$), and $H$ is a Gaussian orthogonal/unitary/symplectic or Wishart orthogonal/unitary/symplectic random matrix. The law of $\Gamma$ may be any continuous distribution, which is assumed to be given. We obtain tridiagonal models (in the spirit of Dumitriu--Edelman~\cite{DE}) and compute the joint eigenvalue distribution for any $\beta>0$, not merely $\beta=1,2,4$.

Such ensembles are of active interest in the literature due to the numerous physical applications (see, e.g., the review papers~\cite{FyoSav,NuclearPhys,FyoSom} and references therein).

The problem of computing the exact joint eigenvalue density
of rank one non-Hermitian perturbations of Gaussian ensembles was considered in the physics literature in the papers of Ullah~\cite{Ull69} (for the case $\beta=1$), Sokolov--Zelevinsky~\cite{SokZel} ($\beta=1$), St\"ockmann--\v{S}eba~\cite{StoSeb} ($\beta=1,2$), Fyodorov--Khoruzhenko~\cite{FyoKho99} ($\beta=2$). The present paper provides a rigorous proof of this result (e.g., none of these papers addressed the question of the space of all attainable configurations of eigenvalues, which can be subtle, see below the case for Laguerre ensembles). Moreover, we obtain a generalization for any $\beta>0$ and for any continuous distribution of $\Gamma$.

Let us also mention that the asymptotic analysis of these perturbations are also of high interest in the mathematics and physics literature and have been studied in~\cite{FyoSom96,FyoSom97,FyoSom,SFT99}, see also~\cite{ORoWoo,Roc15}.

The joint eigenvalue density for rank one non-Hermitian perturbations of Wishart (Laguerre) ensembles has not appeared before neither in the mathematics nor physics literature. We treat all cases of $\beta>0$, $m$, and $n$ (we stress that cases $m < n$ and $m > n$ have drastically different behaviours here), and $\Gamma$.

%
%
%
%

The current paper is the Hermitian counterpart of the unitary results from~\cite{KK_truncations} (joint work with R.~Killip). The important cornerstones in the proof are the Dumitriu--Edelman matrix models ~\cite{DE}, and Arlinski{\u\i}--Tsekanovski{\u\i}'s  spectral analysis of (deterministic) Jacobi matrices with rank one imaginary part~\cite{AT}.

We note that our methods can provide matrix models (namely, \textit{block} Jacobi matrices with independent (matrix-valued) Jacobi coefficients) for higher order perturbations $s\ge 2$ as well, which could prove to be useful for computing their eigenvalue density (for the case  $\beta=2$, $s\ge2$, Fyodorov--Khoruzhenko~\cite{FyoKho99} provide another approach). We leave this as a challenging open problem.

\smallskip

\textbf{Acknowledgements}: It is a pleasure to thank Rowan Killip, Yan Fyodorov, Boris Khoruzhenko,  and Dmitry Savin 
for useful discussions and help with the references. The majority of work was done during the author's stay at the Royal Institute of Technology (Stockholm). The author is grateful to the Department of Mathematics, and especially Kurt Johansson, for the hospitality.

\section{Preliminaries}
\subsection{Gaussian and Wishart ensembles}


\begin{definition}
We say that a real-valued random variable $($r.v.$)$ $\xi$ is $N (0,\sigma)$-distributed, and we write $\xi\sim N (0,\sigma)$,  if its probability distribution function $($p.d.f.$)$ is $\tfrac{1}{\sqrt{2\pi \sigma^2}} e^{-x^2/2\sigma^2}$.


We say that a complex-valued r.v. $\xi$ is $N (0,\sigma \mathbf{I}_2)$-distributed $($where $\mathbf{I}_k$ is the $k
\times k$ identity matrix$)$ if $\re \xi$ and $\im \xi$ are independent and each distributed according to $N (0,\sigma)$.


We say that a quaternion-valued r.v. $\xi$ is $N(0,\sigma \mathbf{I}_4)$-distributed if $\xi = \xi_1 + \xi_2 \mathsf{i} + \xi_3 \mathsf{j} + \xi_4 \mathsf{k}$ and $\xi_1,\ldots,\xi_4$ are independent and each distributed according to $N (0,\sigma)$.

We say that a real-valued r.v. $\xi$ is $\chi^2_k$ distributed $(k>0)$ if its p.d.f. is $\tfrac1{2^{k/2}\Gamma(k/2)} x^{k/2-1}e^{-x/2}$. For integer $k>0$ this can be realized by the sum of squares of $k$ independent $N(0,1)$ variables.

We say that a real-valued r.v. $\xi$ is $\chi_k$ distributed $(k>0)$ if it can be realized as the square root of a $\chi^2_k$ random variable. Its p.d.f. is $\tfrac{2^{1-k/2}}{\Gamma(k/2)} x^{k-1}e^{-x^2/2}$. 

We say that a real-valued r.v. $\xi$ is $\tilde\chi_k$ distributed $(k>0)$ if its p.d.f. is $\tfrac{2}{\Gamma(k/2)} x^{k-1}e^{-x^2}$ $($this coincides with $\tfrac{1}{\sqrt2} \chi_k$ distribution$)$.

\end{definition}

\begin{definition}\label{gue}
Let $Y$ be an $n\times n$ matrix with independent identically distributed $($i.i.d.$)$ entries chosen from $N(0,1)$, $N(0, \mathbf{I}_2)$, or $N(0, \mathbf{I}_4)$. 
Then we say that $X=\tfrac{1}{2} (Y + Y^*)$ belongs to the Gaussian orthogonal/unitary/symplectic ensemble, respectively. We denote it by $GOE_n$, $GUE_n$, $GSE_n$, respectively.
\end{definition}

\begin{definition}\label{lue}
Let $Y$ be an $m\times n$ matrix with i.i.d. entries chosen from $N(0,1)$, $N(0, \mathbf{I}_2)$, or $N(0, \mathbf{I}_4)$.  
Then we say that the $n\times n$ matrix $X=Y^* Y$ belongs to the Wishart orthogonal/unitary/symplectic ensemble, respectively. We denote it by $LOE_{(m,n)}$, $LUE_{(m,n)}$, $LSE_{(m,n)}$, respectively.
\end{definition}

To avoid confusion, we stress that $LOE_{(m,n)}/LUE_{(m,n)}/LSE_{(m,n)}$ ensembles consist of $n\times n$ matrices. 

\subsection{Tridiagonalization of Hermitian matrices}\label{ssTridiagonalization}


Let $H$ be an $n \times n$ Hermitian matrix. Let us describe a process that we will call the tridiagonalization procedure. 

Denote $\mathbf{e}_j$ to be the $j$-th standard vector in $\bbC^n$, that is, having $1$ in its $j$-th entry and $0$ everywhere else. Let $\langle \mathbf{x}, \mathbf{y} \rangle := \mathbf{x}^* \mathbf{y}$, the usual inner product in $\bbC^n$.


Let us apply the Gram--Schmidt orthogonalization procedure in $\bbC^n$ to the sequence of vectors $\mathbf{e}_1, H \mathbf{e}_1, H^2 \mathbf{e}_1,\ldots, H^{k-1} \mathbf{e}_1$, where $k= \dim \spann\{ H^j \mathbf{e}_1: j\ge 0 \}$. Note that $1\le k\le n$. After normalization we obtain an orthonormal sequence of vectors $\mathbf{v}_1,\ldots,\mathbf{v}_k$ in $\bbC^n$. If $k< n$, then we choose an arbitrary unit vector $\mathbf{v}_{k+1}$ in $\bbC^n \ominus \spann\{ \mathbf{v}_1,\ldots,\mathbf{v}_k \}$ and repeat the procedure but with $\mathbf{v}_{k+1}$ instead of $\mathbf{e}_1$. By repeating this procedure finitely many times more if necessary and combining all the resulting vectors together, we obtain an orthonormal basis  $\{\mathbf{v}_j\}_{j=1}^n$ of $\bbC^n$.

Standard arguments (see, e.g.,~\cite[Sect 1.3]{Rice}) show that the matrix of $H$ in the basis $\{\mathbf{v}_j\}_{j=1}^n$ is tridiagonal. In other words, if we form unitary matrix $S$ with $\{\mathbf{v}_j\}_{j=1}^n$ as its columns, then $S H S^* = \calJ$, where
\begin{equation}\label{jacobi}
\calJ = S H S^* = \left(
\begin{array}{ccccc}
b_1&a_1&0& &\\
a_1&b_2&a_2&\ddots &\\
0&a_2&b_3&\ddots & 0 \\
 &\ddots&\ddots&\ddots & a_{n-1} \\
 & & 0 & a_{n-1} & b_n
\end{array}\right), \quad a_j \ge 0, b_n\in\bbR.
\end{equation}
We call matrices of the form~\eqref{jacobi} Jacobi, and the coefficients $\{a_j,b_j\}$ --- their Jacobi coefficients. For a future reference, observe that
\begin{equation}\label{fixE1}
S \mathbf{e}_1 = S^* \mathbf{e}_1 =  \mathbf{e}_1
\end{equation}
since $\mathbf{v}_1 = \mathbf{e}_1$ in the Gram--Schmidt procedure. Note that in the tridiagonalization procedure above, if $\dim \spann\{ H^j \mathbf{e}_1: j\ge 0 \} = k < n$, then $a_k = 0$, i.e., $\calJ$ becomes a direct sum of Jacobi matrices.

\subsection{Matrix models for Gaussian and Wishart ensembles}\label{ssDE}
Now let us apply the tridiagonalization procedure from the previous section to a random matrix from a Gaussian or a Wishart ensemble.

If $H$ is from $GOE_n$, $GUE_n$, or $GSE_n$, then $\mathbf{e}_1$ is a cyclic vector for $H$ with probability $1$. Therefore we obtain~\eqref{jacobi} with $a_j>0$ for all $1\le j\le n-1$.

The same is true for a random matrix $H$ from $LOE_{(m,n)}$, $LUE_{(m,n)}$, or $LSE_{(m,n)}$, but only if $m \ge n$. If $m<n$, then  with probability $1$,  $\dim \spann\{ H^j \mathbf{e}_1: j\ge 0 \} = m +1 \le n$ , and $\bbC^n \ominus \spann\{ H^j \mathbf{e}_1: j\ge 0 \} \subseteq \ker H$, so that the resulting Jacobi matrix~\eqref{jacobi} that we obtain has $a_{m+1} = \ldots = a_{n-1} = 0$, $b_{m+2} = \ldots = b_n=0$. In other words, we have that $\calJ$ is the direct sum of an $(m+1)\times (m+1)$ Jacobi matrix and the $(n-m-1)\times (n-m-1)$ zero matrix. The proof of this case can be done by following the Dumitriu--Edelman~\cite{DE} arguments.

\begin{lemma}[Dumitriu--Edelman~\cite{DE}]\label{DE1}
Let $H$ be a $GOE_n$, $GUE_n$, or $GSE_n$ matrix. 
There exists a unitary matrix $S$ satisfying~\eqref{fixE1} such that $S H S^* = \calJ $ is tridiagonal~\eqref{jacobi}, where
\begin{align*}
a_j & \sim \tilde\chi_{\beta(n-j)},  & 1\le j\le n-1, \\
b_j & \sim N(0,1),  & 1\le j\le n,
\end{align*}
where $\beta =1,2,4$ for $GOE_n$, $GUE_n$, $GSE_n$, respectively.
\end{lemma}

\begin{lemma}[Dumitriu--Edelman~\cite{DE}]\label{DE2}
Let $H$ be a $LOE_{(m,n)}$, $LUE_{(m,n)}$, or $LSE_{(m,n)}$ matrix. 
There exists a unitary matrix $S$ satsifying~\eqref{fixE1} such that $S H S^* = \calJ= B^* B$ is tridiagonal~\eqref{jacobi}, where
\begin{equation}\label{bidiag}
B = \left(
\begin{array}{ccccc}
x_1&y_1&0& &\\
0&x_2&y_2&\ddots &\\
0&0&x_3&\ddots & 0 \\
 &\ddots&\ddots&\ddots & y_{n-1} \\
 & & 0 & 0 & x_n
\end{array}\right), \quad \mbox{with }
\end{equation}
(i) If $m \ge n$:
\begin{align*}
x_j & \sim \chi_{\beta(m-j+1)},  & 1\le j\le n, \\
y_j & \sim \chi_{\beta(n-j)},  & 1\le j\le n-1;
\end{align*}
(ii) If $m \le  n-1$:
\begin{align*}
x_j & \sim \begin{cases} \chi_{\beta(m-j+1)},  & \mbox{if } 1\le j\le m,   \\ 0 ,  & \mbox{if } m+1\le j\le n, \end{cases} \\
y_j & \sim \begin{cases} \chi_{\beta(n-j)},  & \mbox{if } 1\le j\le m,  \\  0 ,  & \mbox{if } m+1\le j\le n-1 ; \end{cases}
\end{align*}
where $\beta =1,2,4$ for $LOE_{(m,n)}$, $LUE_{(m,n)}$, $LSE_{(m,n)}$, respectively.
\end{lemma}
\begin{remarks}
1. For $GSE_n$ and $LSE_{(m,n)}$ every entry is quaternionic, so all the instances of $\bbC$ in the arguments above should be replaced with the algebra of quaternions. The resulting coefficients $a_j$, $b_j$, $x_j$, $y_j$ in Lemmas~\ref{DE1}, \ref{DE2} are quaternionic too, but with the $\mathsf{i}$, $\mathsf{j}$, and $\mathsf{k}$ parts equal to zero. 

2. We adopt a different notation from the one used in~\cite{DE}: the roles of $a_j$'s and $b_j$'s are switched; the orderings of $a_j$, $b_j$, $x_j$, $y_j$ have been reversed; Wishart ensembles are taken to be $W^* W$ instead of $W W^*$.
\end{remarks}

\subsection{Gaussian and Laguerre $\beta$-ensembles}\label{ssBetaEnsembles}
The tridiagonal matrix ensembles from Lemmas~\ref{DE1} and~\ref{DE2} make sense for any $0<\beta<\infty$, not merely for $\beta=1,2,4$. We will call them Gaussian $\beta$-ensembles $G\beta E_n$ and Laguerre $\beta$-ensembles $L\beta E_{(m,n)}$, respectively.

%

\subsection{Spectral measures of Gaussian and Laguerre $\beta$-ensembles}
By the Riesz representation theorem, for any Hermitian matrix $H$ there exists a probability measure $\mu$ satisfying
\begin{equation}\label{spectralT}
\langle \mathbf{e}_1, H^k  \mathbf{e}_1 \rangle = \int_\bbR x^k d\mu(x), \quad \mbox{for all } k\ge 0.
\end{equation}
We call $\mu$ the spectral measure of $H$ corresponding to the vector $\mathbf{e}_1$.

In fact, any Hermitian can be unitarily diagonalized, so that we can write $H = U D U^*$, where $D$ is the diagonal matrix with eigenvalues $\lambda_1,\ldots,\lambda_n$ of $H$ on the diagonal, and the columns $\textbf{u}_1,\ldots,\textbf{u}_n$ of $U$ are the corresponding orthonormal eigenvectors of $H$. This easily implies~\eqref{spectralT} with
\begin{equation}\label{spectralM}
\mu(x) = \sum_{j=1}^n w_j \delta_{\lambda_j} , \quad \mbox{ where } w_j = |\langle \mathbf{e}_1, \mathbf{u}_j \rangle|^2.
\end{equation}
Here $\delta_\lambda$ is the Dirac measure at $\lambda$, i.e., the probability measure concentrated at a point $\lambda$. Note that the support of $\mu$ consists of $\le n$ points ($<n$ if some of the eigenvalues coincide or if some of the eigenvectors are orthogonal to $\mathbf{e}_1$).

Note that because of~\eqref{fixE1}, the spectral measures of $H$ and of its Jacobi form $\calJ$ coincide, that is $H$ and $\calJ$ have identical eigenvalues $\lambda_j$'s and eigenweights $w_j$'s. In particular, spectral measures of $GOE_n$ and $G \beta E_n$ with $\beta=1$ coincide; spectral measures of $GUE_n$ and $G \beta E_n$ with $\beta=2$ coincide; \textit{quaternion-valued} spectral measures of $GSE_n$ and $G \beta E_n$ with $\beta=4$ (viewed as a matrix with purely-real quaternion entries) coincide. Analogous statements can be made for $LOE_{(m,n)}$/$LUE_{(m,n)}/LSE_{(m,n)}$ and $L\beta E_{(m,n)}$. 

We remark that all the statements in lemmas and theorems below should be understood to hold with probability 1.

\begin{lemma}[Dumitriu--Edelman~\cite{DE}]\label{DE3}
For any $0<\beta<\infty$, the spectral measure of $G \beta E_n$-ensemble is~\eqref{spectralM} where $\lambda_1,\ldots,\lambda_n, w_1,\ldots, w_{n-1}$ are distributed on
\begin{equation}\label{domain}
 \sum_{j=1}^n w_j =1;  \quad w_j > 0, \quad 1 \le j \le n; \quad \lambda_j \in \bbR
\end{equation}
according to
\begin{equation}\label{evGaussian}
\tfrac{1}{g_{\beta,n}} \prod_{j=1}^n e^{-\lambda_j^2/2} \prod_{1\le j<k \le n} |\lambda_j - \lambda_k|^\beta d\lambda_1 \ldots d\lambda_n
\times \tfrac{1}{c_{\beta,n}} \prod_{j=1}^n w_j^{\beta /2 -1 }dw_1 \ldots dw_{n-1},
\end{equation}
where
\begin{equation}\label{normalizations}
g_{\beta,n} = (2\pi)^{n/2} \prod_{j=1}^{n} \frac{\Gamma(1+\beta j/2)}{\Gamma(1+\beta/2)},
\quad
c_{\beta,n} = \frac{\Gamma(\beta/2)^n}{\Gamma(\beta n/2)}.
\end{equation}
\end{lemma}

\begin{lemma}[Dumitriu--Edelman~\cite{DE}]\label{DE4}
For any $m \ge n$ and any $0<\beta<\infty$, the spectral measure of $L \beta E_{(m,n)}$-ensemble is~\eqref{spectralM} where $\lambda_1,\ldots,\lambda_n, w_1,\ldots, w_{n-1}$ are distributed on
\begin{equation}\label{domain2}
 \sum_{j=1}^n w_j =1;  \quad w_j > 0, \quad 1 \le j \le n; \quad \lambda_j >0
\end{equation}
according to
\begin{equation}\label{evLaguerre}
\tfrac{1}{l_{\beta,n,a}} \prod_{j=1}^n \lambda_j^{\beta a/2} e^{-\lambda_j/2} \prod_{1\le j<k \le n} |\lambda_j - \lambda_k|^\beta d\lambda_1 \ldots d\lambda_n
\times \tfrac{1}{c_{\beta,n}} \prod_{j=1}^n w_j^{\beta /2 -1 }dw_1 \ldots dw_{n-1},
\end{equation}
where $a=|m-n|+1-2/ \beta$, and
\begin{equation}\label{normalizations2}
l_{\beta,n,a} = 2^{n(a\beta/2 +1+(n-1)\beta/2)} \prod_{j=1}^{n} \frac{\Gamma(1+\beta j /2) \Gamma(1+\beta a/2+\beta(j-1)/2)}{\Gamma(1+\beta/2)},
\end{equation}
and $c_{\beta,n}$ is as in~\eqref{normalizations}.
\end{lemma}

\begin{proposition}\label{Pr1}
For any $m \le n-1$ and any $0<\beta<\infty$, the spectral measure of $L \beta E_{(m,n)}$-ensemble is
\begin{equation}\label{spectralM2}
\mu(x) = w_0 \delta_0 + \sum_{j=1}^m w_j \delta_{\lambda_j} , 
\end{equation}
where $\lambda_1,\ldots,\lambda_m, w_1,\ldots, w_{m}$ are distributed on
\begin{equation}\label{domain3}
\sum_{j=0}^m w_j =1;  \quad w_j > 0, \quad 0 \le j \le m; \quad \lambda_j >0
\end{equation}
according to
\begin{multline}\label{evLaguerre2}
\tfrac{1}{l_{\beta,m,a}} \prod_{j=1}^m \lambda_j^{\beta a/2} e^{-\lambda_j/2} \prod_{1\le j<k \le m} |\lambda_j - \lambda_k|^\beta d\lambda_1 \ldots d\lambda_m
\\
\times \tfrac{1}{d_{\beta,m,n}} w_0^{\beta(n-m) /2 -1} \prod_{j=1}^m w_j^{\beta /2 -1 }dw_1 \ldots dw_{m},
\end{multline}
where $a=|n-m|+1-2/ \beta$; $l_{\beta,m,a}$ is as in~\eqref{normalizations2}; and
\begin{equation}\label{normalizations3}
d_{\beta,m,n} = \frac{\Gamma(\beta(n-m)/2) \Gamma(\beta/2)^m}{\Gamma(\beta n/2)}.
\end{equation}
\end{proposition}
\begin{proof}
Let us first deal with $\beta =1$ case, which by the discussion before Lemma~\ref{DE1} reduces to computing the spectral measures of a matrix $H$ from $LOE_{(m,n)}$. For this ensemble, the eigenvalue distribution is as stated, since the nonzero eigenvalues of $LOE_{(m,n)}$ are distributed identically to the eigenvalues of $LOE_{(n,m)}$.

With probability one, we may assume that eigenvalues of $H$ satisfy $\lambda_1>\ldots>\lambda_m>0=0=\ldots=0$ ($n-m$ zeros).
Let us choose an orthonormal system of (real) eigenvectors $\mathbf{u}_1,\ldots,\mathbf{u}_n$ of $H$ corresponding to these eigenvalues, respectively. We pick each $\mathbf{u}_j$ at random uniformly from the set of all possible choices. 
Since for any $n\times n$ orthogonal matrix $O$, the matrix $O^T H O$ also belongs to  $LOE_{(m,n)}$, we can see that: $\mathbf{u}_1$ is uniformly distributed on the unit sphere $\{\mathbf{u}\in \bbR^n: ||\mathbf{u}|| = 1 \}$; and for any $1\le j \le n$, the vector $\mathbf{u}_j$ conditionally on $\mathbf{u}_1, \ldots, \mathbf{u}_{j-1}$ is uniformly distributed on the subset of this unit sphere that is orthogonal to $\mathbf{u}_1, \ldots, \mathbf{u}_{j-1}$. Therefore the matrix consisting of the eigenvectors as its columns is a Haar distributed orthogonal matrix (see, e.g.,~\cite[Prop~2.2(a)]{KK_truncations}). Then its first row $(x_1,\ldots,x_n)$ is distributed uniformly on the unit sphere $\{\mathbf{u}\in \bbR^n: ||\mathbf{u}|| = 1 \}$. Now recalling~\eqref{spectralM}, we obtain that $w_j = x_j^2$, $1\le j \le m$, and $w_0 = x_{m+1}^2+\ldots+ x_n^2$.
Now one can apply arguments from the proof of~\cite[Cor A.2]{KN} (note that $dw_j = 2w_j^{1/2} dx_j$) to see that the joint distribution of $w_1,\ldots,w_m$ is proportional to $w_0^{(n-m-2)/2} \prod_{j=1}^m w_j^{-1/2} dw_1 \ldots dw_m$. Let us ignore the normalization constant for now and come back to it in the end.

Just as in Dumitriu--Edelman~\cite{DE}, this allows us to compute the Jacobian of the change of variables from $\{x_j,y_j\}_{j=1}^m$ in~\eqref{bidiag} to $\{\lambda_j,w_j\}_{j=1}^m$. Before proceeding, we need to clarify why this change of variables is bijective. By Favard's theorem (see, e.g.,~\cite[Thms~1.3.2--1.3.3]{Rice}), there is one-to-one correspondence between all $(m+1) \times (m+1)$ Jacobi matrices~\eqref{jacobi} with all $a_j>0$ ($1\le j\le m$) and all probability measures supported on $m+1$ distinct points. This trivially implies that there is one-to-one correspondence between all positive semi-definite $(m+1) \times (m+1)$ Jacobi matrices $\calJ$ with $\det\calJ=0$ and $a_j>0$ (for each $1\le j\le m$) and all probability measures supported on $m+1$ points of the form~\eqref{spectralM2}--\eqref{domain3}. By semi-definiteness, any such $\calJ$ can be Cholesky factorized $\calJ = B^* B$ with $B$ upper-triangular with non-negative entries on the diagonal. Since $\calJ$ is tridiagonal, it is not hard to see that this $(m+1)\times (m+1)$ matrix $B$ must be two-diagonal as in~\eqref{bidiag} with $x_j\ge 0$, $1\le j \le m+1$. Since $\det\calJ=0$, we must have that $x_j=0$ for at least one $1\le j\le m+1$. But since all $a_j>0$, we must necessarily have $x_{m+1}=0$, and $x_j\ne 0$ for $1\le j\le m$. $a_j>0$ also implies that $y_j>0$, $1\le j\le m$. Conversely, any $(m+1)\times(m+1)$ matrix $B$ of the form~\eqref{bidiag} with $x_j>0, y_j>0$ for $1\le j\le m$ and $x_{m+1}=0$ clearly leads to a positive semi-definite $(m+1) \times (m+1)$ Jacobi matrix $\calJ$ with $\det\calJ=0$ and all $a_j>0$ ($1\le j\le m$).

Using the matrix model in Lemma~\ref{DE2} (case $m<n$) and the distribution~\eqref{evLaguerre2} that we proved for $\beta =1$, we obtain that the Jacobian is proportional (let us ignore the normalizing constants for now) to
\begin{multline*}
\det \frac{\partial(x_1,\ldots,x_m,y_1,\ldots,y_m)}{\partial(\lambda_1,\ldots,\lambda_m,w_1,\ldots,w_m)}
\propto
\prod_{j=1}^m x_j^{-m+j} e^{x_j^2/2} \prod_{j=1}^m y_j^{-n+j+1} e^{y_j^2/2}
\\
\times
w_0^{\tfrac{n-m}{2} -1} \prod_{j=1}^m w_j^{-\tfrac12 }
 \prod_{j=1}^m \lambda_j^{\tfrac{n-m-1}{2}} e^{-\tfrac{\lambda_j}{2}} \prod_{1\le j<k\le m} |\lambda_j - \lambda_k|
\end{multline*}
Now, recall Lemma~\ref{DE2}. The joint distribution of $\{x_1,\ldots,x_m,y_1,\ldots,y_m\}$ for $L \beta E_{(m,n)}$, $m\le n-1$, is, up to a normalizing constant,
$$
\propto \prod_{j=1}^m x_j^{\beta(m-j-1)-1} e^{-x_j^2/2} dx_j \, \prod_{j=1}^m y_j^{\beta(n-j)-1} e^{-y_j^2/2} dy_j.
$$
Using the above Jacobian, we obtain that this distribution becomes
\begin{multline}\label{distrInter}
\propto
\prod_{j=1}^m x_j^{(\beta-1)(m-j-1)} \, \prod_{j=1}^m y_j^{(\beta-1)(n-j)}
\\ \times
w_0^{\tfrac{n-m}{2} -1} \prod_{j=1}^m w_j^{-\tfrac12 }
 \prod_{j=1}^m \lambda_j^{\tfrac{n-m-1}{2}} e^{-\tfrac{\lambda_j}{2}} \prod_{1\le j<k\le m } |\lambda_j - \lambda_k|
\, d\lambda_1\ldots d\lambda_m dw_1\ldots dw_m.
\end{multline}

\begin{lemma}
\begin{itemize}
\item[(i)] The following identity holds
$$
\prod_{j=1}^m x_j^{m-j+1} y_j^{m-j+1} = \prod_{j=0}^m w^{1/2}_j \prod_{1\le j<k\le m} |\lambda_j-\lambda_k| \prod_{j=1}^m \lambda_j  .$$
\item[(ii)] The following identity holds
$$
\prod_{j=1}^m y_j^2 = w_0 \prod_{j=1}^m \lambda_j .
$$
\end{itemize}
\end{lemma}
\begin{proof}
(i) follows immediately by noting that $x_j y_j = a_j$, $1\le j \le m$, and then applying~\cite[Lemma 2.7]{DE}. Note the clash of notations: their $n$ is our $m+1$, their $\{b_1,\ldots,b_{n-1}\}$, $\{\lambda_1,\ldots,\lambda_n\}$, and $\{q_1^2,\ldots,q_n^2\}$ are ours $\{a_m,\ldots,a_1\}$, $\{\lambda_1,\ldots,\lambda_m,0\}$, and $\{w_1,\ldots,w_m,w_0\}$, respectively.

To prove (ii), we use theory of orthogonal polynomials, see, e.g.,~\cite{Rice}. By combining~\cite[Prop 3.2.8]{Rice} and~\cite[Prop 2.3.12]{Rice} we get
$$
w_0 = -\lim_{z\to 0}  \langle \mathbf{e}_1, z (\calJ-z)^{-1} \mathbf{e}_1 \rangle  = \lim_{z\to 0} \tfrac{z q_{m+1}(z)}{p_{m+1}(z)} =  \tfrac{q_{m+1}(0)}{p'_{m+1}(0)},
$$
where $p_j$'s and $q_j$'s are the orthonormal polynomials associated to $\calJ$ of the first and second kind, respectively (in order to define $p_{m+1}$ and $q_{m+1}$ we need $a_{m+1}$ which we take to be an arbitrary positive number).  By~\cite[Thm 1.2.4]{Rice}, $p_{m+1}(z) = (\prod_{j=1}^{m+1} a_j^{-1}) \det(z-\calJ)$, so $p'_{m+1}(0) = (-1)^m \prod_{j=1}^{m+1} a_j^{-1} \prod_{j=1}^m \lambda_j$. Using  the Wronskian relation~\cite[Prop 3.2.3]{Rice} and $p_{m+1}(0) = 0$ (since $0$ is an eigenvalue of $\calJ$), we obtain $q_{m+1}(0)= 1/(a_{m+1} p_{m}(0))$. Finally, $p_m(z) = (\prod_{j=1}^{m} a_j^{-1}) \det(z-\calJ_{m\times m})$, where $\calJ_{m\times m}$ is the $m\times m$ top left corner of $\calJ$. Recall that $\calJ = B^* B$. It is easy to see that $\calJ_{m\times m} = B_{m\times m} ^* B_{m\times m}$, where $B_{m\times m}$ is the $m\times m$ top left corner of $B$. Therefore $p_m(0)=(\prod_{j=1}^{m} a_j^{-1}) \det(-B_{m\times m} ^* B_{m\times m}) = (-1)^m (\prod_{j=1}^{m} a_j^{-1}) \prod_{j=1}^m x_j^2$. Combining this all together with $a_j = x_j y_j$, $1\le j \le m$, we obtain (ii).
\end{proof}
With the aid of this lemma we can now simplify the distribution~\eqref{distrInter}. Indeed, using the identity (i) we can eliminate the product of involving $x_j$'s, and then using (ii), we can eliminate the product involving $y_j$'s. It is an easy exercise to see that, up to a normalization constant, this reduces~\eqref{distrInter} to~\eqref{evLaguerre2}. Finally, note that $l_{\beta,m,a}$ is the right normalization constant for the eigenvalues in~\eqref{evLaguerre2} by Lemma~\ref{DE4}. And the normalization constant $d_{\beta,m,n}$ can be computed by evaluating the Dirichlet integral, see, e.g.,~\cite[Cor A.4]{KN}.
\end{proof}

\section{Rank one perturbations}\label{sRankOne}
Let $H$ be an $n \times n$ matrix from one of the six ensembles $GOE_n$, $LOE_{m\times n}$ (let us refer to these two ensembles as the $\beta=1$ case throughout this section); $GUE_n$,  $LUE_{m\times n}$ ($\beta=2$ case); $GSE_n$, $LSE_{m\times n}$ ($\beta=4$ case). Let
\begin{equation}\label{effective}
H_{eff}=H + i \Gamma,
\end{equation}
where $\Gamma = (\Gamma_{jk})_{j,k=1}^n$ is an $n\times n$ positive matrix that is independent of $H$ with real 
(if $\beta=1$), complex 
(if $\beta=2$), or quaternion  
(if $\beta=4$) entries. We assume that $\Gamma$ has rank $1$ (for the case $\beta=4$, the (right) rank is viewed over quaternions, see, e.g.,~\cite{bRodman}). 


Since $\Gamma$ is Hermitian, we can diagonalize $\Gamma = U^* (l I_{1\times1}) U$, where $I_{1\times1}$ is the $n\times n$ matrix with $(1,1)$-entry equal to $1$ and $0$ everywhere else, and $U$ is orthogonal, unitary, or unitary symplectic for $\beta=1,2,4$, respectively (for quaternion diagonalization, see, e.g.,~\cite[Thm 5.3.6]{bRodman}). Since the Hilbert--Schmidt norm should be preserved, we see that $l=||\Gamma||_{HS} = (\sum_{j,k=1}^n |\Gamma_{jk}|^2)^{1/2}$.

Then
$
H_{eff}=U^*( UHU^* + i l I_{1\times1} ) U,
$
where $U$ is independent of $H$.
From Definitions~\ref{gue} and~\ref{lue}, it is clear that $UHU^*$ belongs to the same ensemble as $H$. 
Therefore we can apply the tridiagonalization procedure from Subsection~\ref{ssTridiagonalization} to reduce $UHU^*$ to the Dumitriu--Edelman form: $UHU^* = S^* \calJ S$ with $\calJ$ as in Lemmas~\ref{DE1} or~\ref{DE2}, and $S$ unitary  with $S \mathbf{e}_1 = S^* \mathbf{e}_1 = \mathbf{e}_1$. This implies $S I_{1\times1} S^*=I_{1\times1} $ and therefore
$$
H_{eff}=U^*S^* ( \calJ + i l I_{1\times1} ) S U.
$$
This shows that $H_{eff}$ can be unitarily reduced to a tridiagonal form whose all entries are real, except for the complex $(1,1)$ entry. We can formalize it into a theorem.

\begin{theorem}[Matrix model for rank one non-Hermitian perturbations of Gaussian and Wishart ensembles] \label{thmModels}
Let $H$ be taken from one of the six ensembles $GOE_n$,  $GUE_n$, $GSE_n$, $LOE_{m\times n}$ ,  $LUE_{m\times n}$ ,  $LSE_{m\times n}$. Suppose $\Gamma \ge \mathbf{0}$, $\rank \Gamma = 1$, and $\Gamma$ is independent of $H$. Then $H_{eff} = H+ i \Gamma$ is unitarily equivalent to
\begin{equation}\label{modelPert}
\calJ + i l I_{1\times1}
\end{equation}
where $\calJ$ is as in Lemma~\ref{DE1} or~\ref{DE2}, respectively, and $l=||\Gamma||_{HS} = (\sum_{j,k=1}^n |\Gamma_{jk}|^2)^{1/2}$ is independent of $\calJ$.
\end{theorem}

\begin{remark}
1. Just like Dumitriu--Edelman models, this tridiagonal matrix ensemble~\eqref{modelPert} makes sense for any $0<\beta<\infty$, not merely $\beta =1,2,4$. For the obvious reasons we will refer to these as non-Hermitian rank one perturbations of $G\beta E_n$ and $L\beta E_{(m,n)}$ ensembles.
\end{remark}

\section{Joint eigenvalue distribution}

For the rest of the paper let
\begin{equation}\label{halfPlane}
\bbC_+ := \{z\in\bbC : \im z>0\}.
\end{equation}

\subsection{Perturbations of Gaussian $\beta$-ensembles}

\begin{theorem}
For any $0<\beta<\infty$, let $\calJ$ be from $G\beta E_n$ ensemble $($see Lemma~\ref{DE1}$)$ and $l$ be independent of $\calJ$ distributed according to an absolutely-continuous probability distribution $F(l)dl$ on $(0,+\infty)$. Then the eigenvalues of~\eqref{modelPert} are distributed on $(\bbC_+)^n$ according to
\begin{multline}\label{zFinal}
\tfrac{1}{h_{\beta,n}} \, e^{-\frac12 \sum_{j=1}^n (\re z_j)^2 - \sum_{j<k} (\im z_j)(\im z_k)}
\\
\times \prod_{j,k=1}^n |z_j-\bar{z}_k |^{\tfrac\beta2 -1} \prod_{j<k} |z_j-z_k|^2  \frac{F(l)}{l^{\tfrac{\beta n}{2}-1}} d^2 z_1\ldots d^2 z_n,
\end{multline}
where $l=\sum_{j=1}^n \im z_j$, $d^2 z$ stands for the 2-dimensional complex Lebesgue measure, and
$$
h_{\beta,n} = 2^{n(\beta/2-1)}  g_{\beta,n} c_{\beta,n},
$$
where $g_{\beta,n}$ and $c_{\beta,n}$ are as in~\eqref{normalizations}.
\end{theorem}
\begin{remark}
In view of Theorem~\ref{thmModels}, distribution~\eqref{zFinal} with $\beta=1,2,4$ is the eigenvalue distribution of rank one perturbations of $GOE_n$, $GUE_n$, $GSE_n$, respectively.
\end{remark}
\begin{proof}
First of all, because the imaginary part of $\calJ$ is positive, we know that each of the eigenvalues $z_1,\ldots,z_n$ lies in $\bbC_+$. The result of Arlinski{\u\i}--Tsekanovski{\u\i}~\cite[Thm 5.1]{AT} says that the mapping
\begin{align}
&\{a_j\}_{j=1}^{n-1}, \{b_j\}_{j=1}^n, l \mapsto z_1,\ldots, z_n \\
&(0,\infty)^{n-1}\times \bbR^n \times (0,\infty) \to (\bbC_+)^n
\end{align}
is one-to-one and onto (up to permutations of $z_j$'s). Then so is the mapping $\{\lambda_j\}_{j=1}^{n}, \{w_j\}_{j=1}^{n-1}, l \mapsto z_1,\ldots, z_n$, where $\mu$~\eqref{spectralM} is the spectral measure of $\calJ$. Let us compute the Jacobian of this transformation.

\begin{lemma}\label{lmJacobian}
\begin{equation}\label{jacobian4}
\left| \det \frac{\partial\left(\re z_1,\ldots, \re z_n,\im z_1,\ldots, \im z_n \right)}{\partial \left(\lambda_1,\ldots,\lambda_n,w_1,\ldots,w_{n-1},l \right)} \right|
= l^{n-1} \prod_{j<k} \frac{|\lambda_j-\lambda_k|^2}{|z_j-z_k|^2}.
\end{equation}
\end{lemma}
\begin{proof}
Denote $\calJ_{l} = \calJ + i l I_{1\times1}$.
Define $m(z)=\langle e_1, (\calJ-z)^{-1} e_1\rangle = \sum_{j=1}^n \frac{w_j}{\lambda_j-z}$. Let $\sum_{j=0}^n \kappa_j z^j = \det(z-\calJ_{l})=\prod_{j=1}^n (z-z_j) $, where $\kappa_n=1$. Then
\begin{multline}\label{charPoly}
\prod_{j=1}^n (z-z_j) = \sum_{j=0}^n \kappa_j z^j  = \det(z-\calJ- il I_{1\times1})  \\
= \det(z-\calJ) \det( I -(z-\calJ)^{-1} il I_{1\times1}) = \left( 1+il m(z)\right) \prod_{j=1}^n (z-\lambda_j).
\end{multline}
By taking the real parts we obtain
\begin{equation}\label{realParts}
\tfrac12 \prod_{j=1}^n (z-z_j)+\tfrac12 \prod_{j=1}^n (z-\bar{z}_j) = \sum_{j=0}^n (\re\kappa_j) z^j = \prod_{j=1}^n (z-\lambda_j),
\end{equation}
which implies
\begin{equation}\label{realPartJac1}
\left| \det\frac{\partial\left(\re \kappa_0,\ldots, \re\kappa_{n-1} \right)}{\partial \left(\lambda_1,\ldots,\lambda_n \right)} \right| = \prod_{j<k} |\lambda_j-\lambda_k|,
\end{equation}
and
\begin{equation}\label{realPartJac2}
\frac{\partial\left(\re \kappa_0,\ldots, \re\kappa_{n-1} \right)}{\partial \left(w_1,\ldots,w_{n-1},l \right)} = \bdnot,
\end{equation}
the $n\times n$ zero matrix. Thus we just need to evaluate $\det\frac{\partial\left(\im \kappa_0,\ldots, \im\kappa_{n-1} \right)}{\partial \left(w_1,\ldots,w_{n-1},l \right)}$, regarding $\lambda_j$'s as constants.

The imaginary parts of~\eqref{charPoly} give
\begin{multline}\label{imagPart}
\sum_{j=0}^{n-1} (\im \kappa_j) z^j = l m(z) \prod_{j=1}^n (z-\lambda_j) = -l \sum_{j=1}^n w_j \prod_{\substack{1 \le k \le n \\ k\ne j}} (z-\lambda_k) \\
=
-l \left[ \sum_{j=1}^{n-1} w_j (\lambda_j-\lambda_n)\prod_{\substack{1 \le k \le n-1 \\ k\ne j}} (z-\lambda_k) \right] - l \prod_{k=1}^{n-1} (z-\lambda_k)
\end{multline}
Denote the polynomial in the square brackets as $s(z)=\sum_{j=0}^{n-2} s_j z^j$. The above equality implies
\begin{equation}\label{jacobian1}
\det\frac{\partial\left(\im \kappa_0,\ldots, \im\kappa_{n-1} \right)}{\partial \left(s_0,\ldots,s_{n-2},l \right)} = (-1)\cdot (-l)^{n-1}.
\end{equation}
Now note that $s(z)$ can  trivially be rewritten as
$$s(z)=\sum_{j=1}^{n-1} \widetilde{w}_j \prod_{\substack{1 \le k \le n-1 \\ k\ne j}} \frac{z-\lambda_k}{\lambda_j-\lambda_k},$$
 where
\begin{equation}\label{wTilde}
\widetilde{w}_j = w_j (\lambda_j-\lambda_n) \prod_{\substack{1 \le k \le n-1 \\ k\ne j}} (\lambda_j-\lambda_k).
\end{equation}
One can now recognize that $s(z)$ is the interpolating polynomial $s(\lambda_k) = \widetilde{w}_k$ for $k=1,\ldots,n-1$. This implies
\begin{equation}\label{jacobian2}
\left| \det\frac{\partial\left(\widetilde{w}_1,\ldots, \widetilde{w}_{n-1} \right)}{\partial \left(s_0,\ldots,s_{n-2} \right)} \right| = \prod_{1 \le j<k \le n-1} |\lambda_j - \lambda_k|.
\end{equation}
Finally, from~\eqref{wTilde},
\begin{equation}\label{jacobian3}
\det\frac{\partial\left(\widetilde{w}_1,\ldots, \widetilde{w}_{n-1} \right)}{\partial \left(w_1,\ldots,w_{n-1} \right)} = \prod_{j=1}^{n-1} (\lambda_j - \lambda_n) \prod_{1 \le j<k \le n-1} |\lambda_j - \lambda_k|^2.
\end{equation}
Combining~\eqref{jacobian1}, \eqref{jacobian2}, \eqref{jacobian3}, we get
\begin{equation}
\left| \det\frac{\partial\left(\im \kappa_0,\ldots, \im\kappa_{n-1} \right)}{\partial \left(w_1,\ldots,w_{n-1},l \right)} \right| = l^{n-1} \prod_{1\le j<k \le n} |\lambda_j-\lambda_k|.
\end{equation}
Using~\eqref{realPartJac1},~\eqref{realPartJac2}, and the fact that the Jacobian of the transformation from  $\re z_1,\ldots, \re z_n,\im z_1,\ldots, \im z_n$ to $\re \kappa_0,\im \kappa_0,\ldots,\re \kappa_{n-1},\im \kappa_{n-1}$ is equal to $\prod_{j<k} |z_j-z_k|^2$, we obtain
\begin{equation}\label{jacobian4}
\left| \det\frac{\partial\left(\re z_1,\ldots, \re z_n,\im z_1,\ldots, \im z_n \right)}{\partial \left(\lambda_1,\ldots,\lambda_n,w_1,\ldots,w_{n-1},l \right)} \right| = l^{n-1} \prod_{j<k} \frac{|\lambda_j-\lambda_k|^2}{|z_j-z_k|^2}.
\end{equation}
\end{proof}

The joint distribution of $\{\lambda_j\}_{j=1}^{n}, \{w_j\}_{j=1}^{n-1}, l$ is
\begin{equation*}
\tfrac{1}{g_{\beta,n} c_{\beta,n}} \prod_{j<k} |\lambda_j-\lambda_k|^\beta \prod_{j=1}^n e^{-  \lambda_j^2/2} \prod_{j=1}^n w_j^{\beta/2-1} F(l) d\lambda_1\ldots d\lambda_n dw_1\ldots dw_{n-1} dl.
\end{equation*}
Using this and the Jacobian computation, we obtain that the distribution of $z_j$'s is
\begin{equation}\label{zPrefinal}
\tfrac{1}{g_{\beta,n} c_{\beta,n}}  \prod_{j<k} |\lambda_j-\lambda_k|^{\beta-2} \prod_{j=1}^n e^{-  \lambda_j^2/2} \prod_{j=1}^n w_j^{\beta/2-1} \frac{F(l)}{l^{n-1}} \prod_{j<k} |z_j-z_k|^2 d^2 z_1\ldots d^2 z_n.
\end{equation}
Note that
\begin{align}
&\label{el} l=-\im \kappa_{n-1} = \sum_{j=1}^n \im z_j, \\
&\label{sum} \sum_{j=1}^n \lambda_j = \sum_{j=1}^n \re z_j, \\
& \sum_{j \ne k} \lambda_j \lambda_k = \sum_{j\ne k} \re (z_j z_k).
\end{align}
The first equation comes from~\eqref{imagPart}, while the latter two follow from~\eqref{realParts}. Then
\begin{equation}\label{lambdaSquared}
\sum_{j=1}^n \lambda_j^2 = \left(\sum_{j=1}^n  \re z_j\right)^2 - \sum_{j\ne k} \re (z_j z_k) = \sum_{j=1}^n (\re z_j)^2 + 2 \sum_{j<k} (\im z_j)(\im z_k).
\end{equation}

Finally, from~\eqref{charPoly},
\begin{equation}\label{eWeights1}
 -i lw_j = il \res_{z=\lambda_j} m(z) = \res_{z=\lambda_j} \prod_{k=1}^n \frac{z-z_k}{z-\lambda_k} =  \frac{\prod_{k=1}^n (\lambda_j-z_k)}{\prod_{k\ne j} (\lambda_j-\lambda_k)},
 \end{equation}
 so
\begin{equation}\label{prodW}
\prod_{j=1}^n w_j = (\tfrac{i}{l})^n  \frac{\prod_{j,k} (\lambda_j-z_k)}{\prod_{j<k} |\lambda_j-\lambda_k|^2} = (\tfrac{i}{l})^n \tfrac{1}{2^n} \frac{\prod_{j,k} (\bar{z}_j-z_k)}{\prod_{j<k} |\lambda_j-\lambda_k|^2} =  \tfrac{1}{(2l)^n} \frac{\prod_{j,k} |\bar{z}_j-z_k|}{\prod_{j<k} |\lambda_j-\lambda_k|^2},
\end{equation}
where we used~\eqref{realParts} with $z=z_k$, $k=1,\ldots,n$. Combining~\eqref{el},~\eqref{lambdaSquared},~\eqref{prodW} with~\eqref{zPrefinal}, we obtain~\eqref{zFinal}.
\end{proof}

\subsubsection*{Examples}
(1) Since $\Gamma$ in Theorem~\ref{thmModels} has rank 1, we can decompose it as $\Gamma = L^* L$, where $L = (l_{1j})_{j=1}^n$ is an $1\times n$ matrix. Assuming the entries $l_{1j}$ of $L$ are independent and normal $N(0,\sigma \mathbf{I}_\beta)$, then $l=\sum_{j=1}^n |l_{1j}|^2 \sim \sigma^2 \chi^2_{\beta n}$, that is $F(l) = \tfrac{1}{(\sqrt2 \sigma)^{\beta n} \Gamma(\beta n/2)} l^{\beta n/2-1}e^{-l/(2\sigma^2)}$. In this special case, distribution~\eqref{zFinal} becomes
\begin{multline}
\tfrac{1}{(\sqrt2 \sigma)^{\beta n} \Gamma(\beta n/2) c_{\beta,n}g_{\beta,n}} \prod_{j,k=1}^n |z_j-\bar{z}_k |^{\beta/2 -1} \prod_{j<k} |z_j-z_k|^2 \\
\times e^{-\frac12 \sum_{j=1}^n (\re z_j)^2 - \sum_{j<k} (\im z_j)(\im z_k)-\tfrac{1}{2\sigma^2} \sum_{j=1}^n \im z_j }  d^2 z_1\ldots d^2 z_n.
\end{multline}
agreeing with the formula obtained by St\"{o}ckman--\v{S}eba~\cite[Eq (4.4)]{StoSeb}.

%

\medskip

(2) If one instead takes, perhaps less naturally, $l\sim \chi_{\beta n/2}$, then the eigenvalue density simplifies to
\begin{equation*}
\propto
\prod_{j,k=1}^n |z_j-\bar{z}_k |^{\beta/2 -1} \prod_{j<k} |z_j-z_k|^2 e^{-\frac12 \sum_{j=1}^n |z_j|^2 - 2\sum_{j<k} (\im z_j)(\im z_k)}  d^2 z_1\ldots d^2 z_n.
\end{equation*}

\subsection{Perturbations of Laguerre $\beta$-ensembles}

Let us first address the question of which eigenvalue configurations are possible for rank one perturbations of Wishart or $\beta$-Laguerre ensembles. 
Unlike the Gaussian case which was easy due to the application of Arlinski{\u\i}--Tsekanovski{\u\i}'s~\cite[Thm 5.1]{AT}, here we perturb a positive (semi-) definite matrix.

\begin{proposition}\label{prConfigurations}
(i)
Let
\begin{equation}\label{jj}
\calJ_l = \calJ + il I_{1\times 1},
\end{equation}
where $l>0$ and $\calJ$ is an $n\times n$ positive definite $($real$)$ Jacobi matrix~\eqref{jacobi} with $a_j>0$, $j=1,\ldots,n-1$. Its eigenvalues, counting algebraic multiplicities, belong to
\begin{equation}\label{configurationSpaceLaguerre1}
\left\{ (z_j)_{j=1}^n \in (\bbC_+)^n :  \sum_{j=1}^n \Arg \,z_j < \tfrac\pi2 \right\}.
\end{equation}
Moreover, for every configuration of $n$ points from~\eqref{configurationSpaceLaguerre1} there exists a unique matrix $\calJ_l$ of the form above with such a system of eigenvalues.

(ii) Let
$$
\calJ_l = \calJ + il I_{1\times 1},
$$
where $l>0$ and $\calJ$ is an $(m+1)\times (m+1)$ positive semi-definite $($real$)$ Jacobi matrix~\eqref{jacobi} with $a_j>0$, $j=1,\ldots,m$, satisfying $\det \calJ = 0$. Its eigenvalues, counting with their algebraic multiplicities, belong to
\begin{equation}\label{configurationSpaceLaguerre2}
\left\{ (z_j)_{j=1}^{m+1} \in (\bbC_+)^{m+1} :  \sum_{j=1}^{m+1} \Arg \,z_j  = \tfrac\pi2 \right\}.
\end{equation}
Moreover, for every configuration of $m+1$ points from~\eqref{configurationSpaceLaguerre2} there exists a unique matrix $\calJ_l$ of the form above with such a system of eigenvalues.
\end{proposition}
\begin{proof}
As before, let $z_j$'s be the eigenvalues of $\calJ_l$; let $\lambda_j$'s and $w_j$'s be the eigenvalues and eigenweights of the spectral measure of $\calJ$ (which is of the form~\eqref{spectralM} with~\eqref{domain2} for the case (i) and~\eqref{spectralM2} with~\eqref{domain3} for the case (ii)). By~\cite{AT}, $z_j\in\bbC_+$ for every $j$.


Consider now case (i). Equations~\eqref{realParts} and~\eqref{imagPart} imply
\begin{align}
\label{elementaryEqual} \re s_k(z_1,\ldots,z_n) &= s_k(\lambda_1,\ldots,\lambda_n), \quad k=1,2,\ldots,n; \\
\label{elementaryEqual2} \im s_k(z_1,\ldots,z_n) &= l \sum_{j=1}^n w_j s_{k-1}(\{\lambda_t\}_{t\ne j}), \quad k=1,2,\ldots,n,
\end{align}
respectively,
where $s_0:=1$, and $s_k$ ($k\ge 1$) is the $k$-th elementary symmetric polynomial
\begin{equation}\label{symmPoly}
s_k(z_1,\ldots,z_n) := \sum_{1\le j_1<j_2<\ldots<j_k\le n} z_{j_1}\ldots z_{j_k}.
\end{equation}
Since for each $j$, $\lambda_j > 0, w_j>0, l>0$, we obtain that $z_1,\ldots,z_n$ must belong to
\begin{equation}\label{configurationSpaceLaguerreAlt}
\left\{ (z_j)_{j=1}^n \in (\bbC_+)^n : s_k(z_1,\ldots,z_n) \in Q_1, \quad k=1,2,\ldots,n \right\},
\end{equation}
where $Q_1:=\{z: 0<\Arg z<\pi/2 \}$.
Conversely, take a sequence of points from~\eqref{configurationSpaceLaguerreAlt}. Since this sequence belongs to $(\bbC_+)^n$, we know from~\cite[Thm 5.1]{AT} that there exists a unique  matrix of the form $\calJ + il I_{1\times 1}$ with $l>0$ and $a_j>0$, $j=1,\ldots,n-1$. We claim that in fact $\calJ$ is positive definite, that is, $\lambda_j>0$ for all $j$. Indeed, equation~\eqref{realParts} along with the positivity of~\eqref{elementaryEqual} implies that $\lambda_1,\ldots,\lambda_n$ are the real roots of the polynomial $\prod_{j=1}^n (z-\lambda_j)$ with alternating signs of the coefficients. By  Descartes' rule of signs, we know such a polynomial cannot have negative zeros. This means that all $\lambda_j$'s are indeed positive. Therefore~\eqref{configurationSpaceLaguerreAlt} is precisely the space of all possible eigenvalue configurations of $H_{eff}$. Let us now show that it coincides with~\eqref{configurationSpaceLaguerre1}.

It is elementary that~\eqref{configurationSpaceLaguerre1} is a subset of~\eqref{configurationSpaceLaguerreAlt}. 
To see the converse, take any sequence from~\eqref{configurationSpaceLaguerreAlt}. Since $ s_n(z_1,\ldots,z_n)=  z_1 z_2 \ldots z_n \in Q_1$, we must have that
\begin{equation}\label{sumArguments}
0+2k\pi< \Arg z_1 + \Arg z_2 +\ldots + \Arg z_n < \pi/2 + 2k\pi
\end{equation}
for some integer $k\ge 0$. We already know that these $z_1,\ldots,z_n$ are the eigenvalues of $\calJ+ il I_1$, where $\calJ$ is \textit{positive definite}. Let us now fix $\calJ$ and view $z_1,\ldots,z_n$ as functions of $l\ge0$ only. Each of these functions is continuous and never passes through $0$. For any $0< l<\infty$, we have~\eqref{sumArguments} for some $k$. But when $l=0$ the sum of the arguments is zero. By continuity $k=0$ for any $l$. This shows that~\eqref{configurationSpaceLaguerreAlt} is a subset of ~\eqref{configurationSpaceLaguerre1}, and therefore they coincide.

To deal with the case (ii), we use similar arguments with $m+1$ instead of $n$ and $\lambda_1,\ldots,\lambda_m,0$ as the eigenvalues (with $\lambda_j>0, j=1,\ldots, m$). With this in mind, equations~\eqref{elementaryEqual} and~\eqref{elementaryEqual2} imply that the eigenvalues $z_1,\ldots,z_{m+1}$ of $\calJ+ilI_{1\times 1}$ belong to
\begin{multline}\label{configurationSpaceLaguerreAlt2}
\big\{ (z_j)_{j=1}^{m+1} \in (\bbC_+)^{m+1} : s_{m+1}(z_1,\ldots,z_{m+1}) \in i\bbR_+; \big. \\
\big. s_k(z_1,\ldots,z_{m+1}) \in Q_1, \quad k=1,2,\ldots,m \big\},
\end{multline}
where $\bbR_+ = \{z\in\bbR: z>0\}$. Conversely,  by~\cite[Thm 5.1]{AT}, any configuration of point from~\eqref{configurationSpaceLaguerreAlt2} coincides with eigenvalues of some $\calJ+ilI_{1\times 1}$, $l>0$. One obtains that the eigenvalues $\lambda_1,\ldots,\lambda_{m+1}$ of $\calJ$ satisfies $s_{k}(\lambda_1,\ldots,\lambda_{m+1})>0$ for $k=1,\ldots,m$ and $s_{m+1}(\lambda_1,\ldots,\lambda_{m+1})=0$. This implies $\lambda_j>0$ for all $j$ except for one zero eigenvalue.

Finally, let us show that~\eqref{configurationSpaceLaguerreAlt2} coincides with~\eqref{configurationSpaceLaguerre2}. The inclusion~\eqref{configurationSpaceLaguerre2}$\subseteq$\eqref{configurationSpaceLaguerreAlt2} is easy. Conversely, take any configuration $\{z_j\}_{j=1}^{m+1}$ from~\eqref{configurationSpaceLaguerreAlt2}. By the above, these points are the eigenvalues of some $\calJ+ilI_{1\times 1}$ with $l>0$, where $\calJ$ has eigenvalues $\{0,\lambda_1,\ldots,\lambda_m\}$ with $\lambda_j>0$ for $1\le j\le m$. Since $s_{m+1} \in i\bbR_+$ in~\eqref{configurationSpaceLaguerreAlt2}, we have
\begin{equation}\label{sumArguments2}
\Arg \, z_1 + \Arg \, z_2 +\ldots + \Arg \,  z_{m+1} = \pi/2 + 2k\pi
\end{equation}
for some integer $k\ge 0$.  After reordering, we can assume that $z_j \to \lambda_j$, $1\le j\le m$, and $z_{m+1}\to 0$ when $l\to0$ (while $\calJ$ is fixed). Therefore $\Arg \, z_j \to 0$ as $l\to 0$ for $1\le j\le m$, while $0\le \Arg \, z_{m+1}\le \pi/2$ for any $l$. This proves that $k=0$, and so~\eqref{configurationSpaceLaguerreAlt2}$\subseteq$\eqref{configurationSpaceLaguerre2}, finishing the proof.
\end{proof}

The following may be known, but if not, it may be of interest on its own. Denote $\bar\bbC_+ := \{z:\im z \ge 0\}$.
\begin{corollary}
Let
$$
H_{eff} = H+ i\Gamma,
$$
where $H$ and $\Gamma$ are positive semi-definite with $\rank \Gamma \le 1$. The eigenvalues of $H_{eff}$, counting with their algebraic multiplicities, belong to
\begin{equation*}
\left\{ (z_j)_{j=1}^n \in (\bar\bbC_+)^n :  \sum_{j=1}^n \Arg \,z_j  \le \tfrac\pi2 \right\},
\end{equation*}
and every such a configuration may occur.
\end{corollary}
\begin{remarks}
1. We stress that this is deterministic result.

2. We adopt the convention $\Arg\, 0 =0$ here.

3. Using our methods one can prove a similar statement for the case when $H$ is not positive-semidefinite, but has $s$ negative eigenvalues. The eigenvalues of $H_{eff}$ then belong to $\big\{ (z_j)_{j=1}^n \in (\bar\bbC_+)^n :   \tfrac\pi2+\pi (s-1) < \sum_{j=1}^n \Arg \,z_j  \le \tfrac\pi2+\pi s \big\}$. 
\end{remarks}
\begin{proof}
Just as in Section~\ref{sRankOne}, we can tridiagonalize $H+i\Gamma= V^* (\calJ +  il I_{1\times 1}) V$, where $V$ is unitary, $l>0$, and $\calJ$ some positive semi-definite tridiagonal $n\times n$ matrix~\eqref{jacobi}. Then just apply the previous proposition. Note that some of the $a_j$'s might be zero which is why we obtain non-strict inequalities in $0\le \Arg \, z_j \le \tfrac\pi2$.
\end{proof}


Now that we know the possible configurations of the eigenvalues, we can compute their joint distribution.

\begin{theorem}
For any $0<\beta<\infty$ and any integer $m,n>0$, let $\calJ$ be the $n\times n$ matrix from $L\beta E_{(m,n)}$ ensemble $($see Subsection~\ref{ssBetaEnsembles}$)$ and $l$ be independent of $\calJ$ distributed according to an absolutely-continuous probability distribution $F(l)dl$ on $(0,+\infty)$.

(i) If $m\ge n$, then the eigenvalues $\{z_1,\ldots,z_n\}$ of $\calJ_l = \calJ + ilI_{1\times 1}$ are distributed on~\eqref{configurationSpaceLaguerre1}
according to
\begin{equation}\label{zFinalLaguerre}
\tfrac{1}{q_{\beta,n,a}} \prod_{j,k=1}^n |z_j-\bar{z}_k |^{\tfrac\beta2 -1} \prod_{j<k} |z_j-z_k|^2 e^{-\frac12 \sum_{j=1}^n \re z_j } \big( \re \prod_{j=1}^n z_j \big)^{\tfrac{\beta a}{2}} \frac{F(l)}{l^{\tfrac{\beta n}{2}-1}} d^2 z_1\ldots d^2 z_n,
\end{equation}
where $l=\sum_{j=1}^n \im z_j$, $a=|m-n|+1-2/\beta$,
and
$$
q_{\beta,n,a} = 2^{n(\beta/2-1)}  l_{\beta,n,a} c_{\beta,n},
$$
where $l_{\beta,n,a}$ and $c_{\beta,n}$ are as in~\eqref{normalizations2} and~\eqref{normalizations}.

(ii) If $m\le n-1$, then eigenvalues of $\calJ_l = \calJ + ilI_{1\times 1}$ are $\{z_1,\ldots,z_{m+1},0,\ldots,0\}$ with $\{z_1,\ldots,z_{m+1}\} =:\{r_1 e^{i\theta_1},\ldots,r_{m+1} e^{i\theta_{m+1}}\}$  distributed on~\eqref{configurationSpaceLaguerre2}
according to
\begin{multline}\label{zFinalLaguerre2}
\tfrac{1}{t_{\beta,m,n}} \prod_{j,k=1}^{m+1} |z_j-\bar{z}_k |^{\tfrac\beta2 -1} \prod_{1\le j<k\le m+1} |z_j-z_k|^2 e^{-\frac12 \sum_{j=1}^{m+1} \re z_j }  \prod_{j=1}^{m+1} |z_j|^{\frac{\beta(n-m-1)}{2}} \frac{F(l)}{l^{\tfrac{\beta n}{2}-1}} \\
dr_1\ldots dr_{m+1} d\theta_1\ldots d\theta_m,
\end{multline}
where $l=\sum_{j=1}^{m+1} \im z_j$  and
\begin{equation}\label{normalizationT}
t_{\beta,m,n} = (m+1) 2^{(m+1)(\beta/2-1)}  l_{\beta,m,a} d_{\beta,m,n},
\end{equation}
where $a=|n-m|+1-2/\beta$, and $l_{\beta,m,a}$ and $d_{\beta,m,n}$ are as in~\eqref{normalizations2} and~\eqref{normalizations3}.
\end{theorem}
\begin{remarks}
1. In view of Theorem~\ref{thmModels}, distributions~\eqref{zFinalLaguerre} and~\eqref{zFinalLaguerre2} with $\beta=1,2,4$ are the eigenvalue distribution of rank one perturbations of the Wishart ensembles $LOE_{(m,n)}$, $LUE_{(m,n)}$, $LSE_{(m,n)}$, respectively.

2. In (ii), $\theta_{m+1} = \pi/2-\sum_{j=1}^m \theta_j$ is implicit due to~\eqref{configurationSpaceLaguerre2}.


\end{remarks}
\begin{proof}
(i) We can take the known joint distribution of the eigenvalues $\lambda_j$'s, eigenweights $w_j$'s (see Lemma~\ref{DE4}), and $l$
and change the variables to $z_j$'s (by Proposition~\ref{prConfigurations}(i) it is one-to-one and onto~\eqref{configurationSpaceLaguerre1}, so the Jacobian~\eqref{jacobian4} applies). Using \eqref{prodW}, \eqref{el}, \eqref{sum}, \eqref{elementaryEqual} (with $k=n$), we obtain the resulting distribution~\eqref{zFinalLaguerre}.

(ii) By Proposition~\ref{prConfigurations}(ii), the map from the spectral measures of the form~\eqref{spectralM2}--\eqref{domain3} to the eigenvalues of $\calJ+il I_{1\times 1}$: $\lambda_1,\ldots,\lambda_m,w_1,\ldots,w_m,l\mapsto z_1,\ldots,z_{m+1}$ is one-to-one and onto~\eqref{configurationSpaceLaguerre2} (if we impose some natural ordering on $\lambda_j$'s and $z_j$'s; we will remove it in the end of the proof). Its Jacobian is of course different from~\eqref{jacobian4} computed earlier. Similar to the notation in the proof of Lemma~\ref{lmJacobian}, let $m(z)=\langle e_1, (\calJ-z)^{-1} e_1\rangle = -\frac{w_0}{z} + \sum_{j=1}^m \frac{w_j}{\lambda_j-z}$ and $\sum_{j=0}^{m+1} \kappa_j z^j = \det(z-\calJ_{l})=\prod_{j=1}^{m+1} (z-z_j) $, where $\kappa_{m+1}=1$. Because of $\det \calJ = 0$, we obtain $\re \kappa_0 = 0$. Following similar reasoning as in the proof of Lemma~\ref{lmJacobian}, we first obtain the value of the Jacobian
\begin{equation}\label{newJac1}
\left|\det \frac{\partial\left(\re \kappa_1,\ldots, \re\kappa_{m},\im\kappa_0,\ldots,\im\kappa_m \right)}{\partial \left(\lambda_1,\ldots,\lambda_m,w_1,\ldots,w_{m},l \right)} \right| = l^m\prod_{j=1}^m \lambda_j \prod_{1\le j<k\le m} |\lambda_j-\lambda_k|^2.
\end{equation}
Now let $z_{j} = r_{j} e^{i\theta_{j}}$ be the polar decomposition of $z_{j}$. Since $\re(z_1\ldots z_{m+1}) = (-1)^{m+1} \re\kappa_0 = 0$, we have that $e^{i\theta_{m+1}}$ is determined by $z_1,\ldots,z_m$. Therefore we have a one-to-one map $\bbR^{2m+1}\to\bbR^{2m+1}$ taking $r_1,\ldots,r_{m+1},\theta_1,\ldots,\theta_{m}$ to $\re \kappa_1,\ldots, \re\kappa_{m},\im\kappa_0,\ldots,\im\kappa_m$. We are headed towards computing its Jacobian on the manifold $\re(z_1\ldots z_{m+1}) =0$.

First off, it is trivial to see that on $\re(z_1\ldots z_{m+1}) =0$:
\begin{equation}\label{newJac2}
\left|\det \frac{\partial\left(\re z_1,\ldots, \re z_m,\im z_1,\ldots,\im z_m, \im \kappa_0 \right)}{\partial \left(r_1,\ldots,r_{m},\theta_1,\ldots,\theta_{m}, r_{m+1} \right)} \right| = \prod_{j=1}^m |z_j|^2.
\end{equation}


We are left with computing the Jacobian of the $\bbR^{2m+1}\to\bbR^{2m+1}$ mapping $\re z_1,\ldots,\re z_m,\im z_1,\ldots,\im z_m, \im \kappa_{0}\mapsto \re \kappa_1,\re \kappa_m,\im \kappa_1,\ldots, \im \kappa_{m}, \im \kappa_0$ restricted to $\re \kappa_0 = 0$, which is easily seen to be equal (cf., e.g.,~\cite[Lemma~D.1]{KK_truncations}) to the Jacobian of the $\bbC^{2m}\times \bbR\to\bbC^{2m}\times \bbR$ map $z_1,\bar z_1,\ldots,z_m,\bar z_m, \im \kappa_{0} \mapsto \kappa_1,\bar \kappa_1,\ldots,\kappa_m,\bar \kappa_m, \im \kappa_{0}$ restricted to $\re \kappa_{0} = 0$, where we treat $z_j$ and $\bar z_j$ as independent variables.
In the notation~\eqref{symmPoly}, we have that $(-1)^{m+1-j}  \kappa_j = s_{m+1-j}(z_1,\ldots,z_{m+1})$ for $0\le j\le m$. Using $\kappa_0 = (-1)^{m+1} z_1\ldots z_{m+1}$ and the trivial equality $s_{m-j}(y_1,\ldots,y_{m}) = s_m(y_1,\ldots,y_{m}) s_{j}(\tfrac{1}{y_1},\ldots,\tfrac{1}{y_{m}}) $, we can write for $1\le j\le m$,
\begin{equation}
(-1)^{m+1-j} \kappa_j =  s_{m+1-j}(z_1,\ldots,z_m) + (-1)^{m+1} \kappa_0 s_{j}(\tfrac{1}{z_1},\ldots,\tfrac{1}{z_m}).
\end{equation}
Since $\bar\kappa_0 = -\kappa_0$, we also get
\begin{equation}
(-1)^{m+1-j} \bar \kappa_j =  s_{m+1-j}(\bar z_1,\ldots,\bar z_m) - (-1)^{m+1} \kappa_0 s_{j}(\tfrac{1}{\bar z_1},\ldots,\tfrac{1}{\bar z_m}).
\end{equation}
These equalities imply
that for $1\le j \le m$ and $1\le s \le m$,
\begin{align*}
\frac{\partial\kappa_j}{\partial z_s} & = (-1)^{m+1-j} s_{m-j}(z_1,\ldots,\hat z_s,\ldots,z_m) - (-1)^j\frac{\kappa_0}{z_s^2} s_{j-1}(\tfrac{1}{z_1},\ldots,\hat{\tfrac{1}{ z_s}},\ldots\tfrac{1}{z_m}), \\
\frac{\partial \bar\kappa_j}{\partial \bar z_s} & = (-1)^{m+1-j} s_{m-j}(\bar z_1,\ldots,\widehat{\bar{z}_s},\ldots,\bar z_m) + (-1)^j\frac{\kappa_0}{{\bar z}_s^2} s_{j-1}(\tfrac{1}{\bar z_1},\ldots,\widehat{\tfrac{1}{ \bar z_s}},\ldots\tfrac{1}{\bar z_m}), \\
\frac{\partial \kappa_j}{\partial \bar z_s} & =  \frac{\partial \bar\kappa_j}{\partial  z_s}  = 0,
\end{align*}
where $\hat v$ means that a variable $v$ is omitted. Using $s_{m-1}(y_1,\ldots, y_{m-1}) s_{j-1}(\tfrac{1}{y_1},\ldots,\tfrac{1}{y_{m-1}}) = s_{m-j}(y_1,\ldots,y_{m-1})$ again, we can rewrite
\begin{align*}
\frac{\partial\kappa_j}{\partial z_s} & =  \frac{z_s-z_{m+1}}{z_s} (-1)^{m+1-j} s_{m-j}(z_1,\ldots,\hat z_s,\ldots,z_m) , \\
\frac{\partial \bar\kappa_j}{\partial \bar z_s} & = \frac{\bar z_s- \bar z_{m+1}}{\bar z_s} (-1)^{m+1-j}   s_{m-j}(\bar z_1,\ldots,\widehat{\bar{z}_s},\ldots,\bar z_m).
\end{align*}
Having this in hand, it is easy to write the Jacobian and perform a straightforward Gaussian elimination to arrive to
\begin{equation}
\left|\det \frac{\partial\left(\kappa_1,\bar \kappa_1,\ldots,\kappa_m,\bar \kappa_m, \im \kappa_0 \right)}{\partial \left(z_1,\bar z_1, \ldots,z_m,\bar z_m,\im\kappa_0 \right)} \right| =  \prod_{j=1}^m |z_j|^{-2} \prod_{1\le j<k\le m+1} |z_j-z_k|^2.
\end{equation}
Combining this with~\eqref{newJac1} and~\eqref{newJac2}, we get
\begin{multline}\label{JacobianLast}
\left| \det\frac{\partial\left(r_1,\ldots,r_{m},\theta_1,\ldots,\theta_{m}, r_{m+1} \right)}{\partial \left(\lambda_1,\ldots,\lambda_m,w_1,\ldots,w_{m},l \right)} \right|
= l^{m} \prod_{j=1}^m |\lambda_j|  \frac{ \prod_{1\le j<k\le m} |\lambda_j-\lambda_k|^2}{\prod_{1\le j<k\le m+1} |z_j-z_k|^2}.
\end{multline}
Repeating the arguments from~\eqref{eWeights1} and~\eqref{prodW}, we obtain
\begin{equation*}
 w_0 = \frac{\prod_{j=1}^{m+1}|z_j|}{l \prod_{j=1}^m|\lambda_j|} ,\quad \mbox{and } \prod_{j=1}^m w_j = \frac{1}{l^{m} 2^{m+1}} \frac{\prod_{j,k=1}^{m+1}|z_j-\bar z_k|}{\prod_{j=1}^{m+1} |z_j| \prod_{j=1}^m |\lambda_j| \prod_{j<k} |\lambda_j-\lambda_k|^2}.
\end{equation*}
Finally, just as in (i), we still have $\sum_{j=1}^m \lambda_j = \sum_{j=1}^{m+1} \re z_j$ and $l=\sum_{j=1}^{m+1} \im z_j$. Now, starting from the joint distribution of $\lambda_1,\ldots,\lambda_m,w_1,\ldots,w_m$ (see Proposition~\ref{Pr1}) and $l$, applying the Jacobian~\eqref{JacobianLast}, and using these substitutions  (note that terms with $\prod |\lambda_j|$ cancel out in the process), we arrive at the distribution~\eqref{zFinalLaguerre2}. Note that the factor $(m+1)$ in~\eqref{normalizationT} comes from removing the ordering of $z_j$'s and $\lambda_j$'s (there are $(m+1)!$ of permutations for $\{z_j\}_{j=1}^{m+1}$, and only $m!$ for $\{\lambda_j\}_{j=1}^m$).
\end{proof}

\subsubsection*{Example}
Choosing $F(l)\propto l^{\tfrac{\beta n}{2}-1} e^{-l/2}$ (as in Example (1) of the previous section, this is natural since corresponds to each entry of $L$ (where $\Gamma = L^* L$) being normal), the distribution~\eqref{zFinalLaguerre} becomes
\begin{equation*}
\propto \prod_{j,k=1}^n |z_j-\bar{z}_k |^{\tfrac\beta2 -1} \prod_{j<k} |z_j-z_k|^2 e^{-\frac12 \sum_{j=1}^n (\re z_j +\im z_j) } \Big( \re \prod_{j=1}^n z_j \Big)^{\tfrac{\beta a}{2}} d^2 z_1\ldots d^2 z_n,
\end{equation*}
and similar simplification can be made for the distribution~\eqref{zFinalLaguerre2}.

\bigskip

\bibliographystyle{amsalpha}
\bibliography{mybib,mybib_rmt}

\def\cprime{$'$} \def\cprime{$'$} \def\cprime{$'$}
  \def\lfhook#1{\setbox0=\hbox{#1}{\ooalign{\hidewidth
  \lower1.5ex\hbox{'}\hidewidth\crcr\unhbox0}}}
\providecommand{\bysame}{\leavevmode\hbox to3em{\hrulefill}\thinspace}
\providecommand{\MR}{\relax\ifhmode\unskip\space\fi MR }
\providecommand{\MRhref}[2]{%
  \href{http://www.ams.org/mathscinet-getitem?mr=#1}{#2}
}
\providecommand{\href}[2]{#2}
\begin{thebibliography}{MRW10}

\bibitem[AT06]{AT}
Yu. Arlinski{\u\i} and E.~Tsekanovski{\u\i}, \emph{Non-self-adjoint {J}acobi
  matrices with a rank-one imaginary part}, J. Funct. Anal. \textbf{241}
  (2006), no.~2, 383--438. \MR{2271925 (2007g:47041)}

\bibitem[DE02]{DE}
I.~Dumitriu and A.~Edelman, \emph{Matrix models for beta ensembles}, J. Math.
  Phys. \textbf{43} (2002), no.~11, 5830--5847. \MR{1936554 (2004g:82044)}

\bibitem[FK99]{FyoKho99}
Y.~V. Fyodorov and B.~A. Khoruzhenko, \emph{Systematic analytical approach to
  correlation functions of resonances in quantum chaotic scattering}, Phys.
  Rev. Lett. \textbf{83} (1999), no.~1, 65--68.

\bibitem[FS96]{FyoSom96}
Y.~V. Fyodorov and H.-J. Sommers, \emph{Statistics of {S}-matrix poles in
  few-channel chaotic scattering: crossover from isolated to overlapping
  resonances}, JETP Letters \textbf{63} (1996), no.~12, 1026--1030.

\bibitem[FS97]{FyoSom97}
\bysame, \emph{Statistics of resonance poles, phase shifts and time delays in
  quantum chaotic scattering: random matrix approach for systems with broken
  time-reversal invariance}, J. Math. Phys. \textbf{38} (1997), no.~4,
  1918--1981, Quantum problems in condensed matter physics. \MR{1450906
  (98k:81297)}

\bibitem[FS03]{FyoSom}
\bysame, \emph{Random matrices close to {H}ermitian or unitary: overview of
  methods and results}, J. Phys. A \textbf{36} (2003), no.~12, 3303--3347,
  Random matrix theory. \MR{1986421}

\bibitem[FS11]{FyoSav}
Y.~V. Fyodorov and D.~V. Savin, \emph{Resonance scattering of waves in chaotic
  systems}, The {O}xford handbook of random matrix theory, Oxford Univ. Press,
  Oxford, 2011, pp.~703--722. \MR{2932654}

\bibitem[KK]{KK_truncations}
R.~Killip and R.~Kozhan, \emph{Matrix models and eigenvalue statistics for
  truncations of classical ensembles of random unitary matrices}, (under
  submission, arXiv:1501.05160).

\bibitem[KN04]{KN}
R.~Killip and I.~Nenciu, \emph{Matrix models for circular ensembles}, Int.
  Math. Res. Not. (2004), no.~50, 2665--2701. \MR{2127367 (2006h:82003)}

\bibitem[MRW10]{NuclearPhys}
G.~E. Mitchell, A.~Richter, and H.~A. Weidenm{\"u}ller, \emph{Random matrices
  and chaos in nuclear physics: nuclear reactions}, Rev. Modern Phys.
  \textbf{82} (2010), no.~4, 2845--2901. \MR{2770945 (2012d:81351)}

\bibitem[OW]{ORoWoo}
S.~O'Rourke and P.~M. Wood, \emph{Spectra of nearly hermitian random matrices},
  (preprint arXiv:1510.00039).

\bibitem[Roc]{Roc15}
J.~Rochet, \emph{Complex outliers of hermitian random matrices}, (preprint
  arXiv:1507.00455).

\bibitem[Rod14]{bRodman}
L.~Rodman, \emph{Topics in quaternion linear algebra}, Princeton Series in
  Applied Mathematics, Princeton University Press, Princeton, NJ, 2014.
  \MR{3241695}

\bibitem[SFT99]{SFT99}
H.-J. Sommers, Y.~V. Fyodorov, and M.~Titov, \emph{{$S$}-matrix poles for
  chaotic quantum systems as eigenvalues of complex symmetric random matrices:
  from isolated to overlapping resonances}, J. Phys. A \textbf{32} (1999),
  no.~5, L77--L87. \MR{1674416}

\bibitem[Sim11]{Rice}
B.~Simon, \emph{Szeg{\H o}'s theorem and its descendants: spectral theory for
  $l{^{2}}$ perturbations of orthogonal polynomials}, M. B. Porter Lectures,
  Princeton University Press, Princeton, NJ, 2011. \MR{2743058 (2012b:47080)}

\bibitem[S{\v{S}}98]{StoSeb}
H.-J. St{\"o}ckmann and P.~{\v{S}}eba, \emph{The joint energy distribution
  function for the {H}amiltonian {$H=H_0-iWW^+$} for the one-channel case}, J.
  Phys. A \textbf{31} (1998), no.~15, 3439--3448. \MR{1625050 (99b:81070)}

\bibitem[SZ89]{SokZel}
V.~V. Sokolov and V.~G. Zelevinsky, \emph{Dynamics and statistics of unstable
  quantum states}, Nuclear Phys. A \textbf{504} (1989), no.~3, 562--588.

\bibitem[Ull69]{Ull69}
N.~Ullah, \emph{On a generalized distribution of the poles of the unitary
  collision matrix}, J. Mathematical Phys. \textbf{10} (1969), 2099--2103.

\end{thebibliography}

\end{document}